\documentclass{amsart}
\usepackage{amsmath,amssymb,amsfonts,amsthm,amscd,indentfirst}

\numberwithin{equation}{section}

\newtheorem{prop}{Proposition}[section]
\newtheorem{theo}[prop]{Theorem}
\newtheorem{lemm}[prop]{Lemma}
\newtheorem{coro}[prop]{Corollary}
\newtheorem{rema}[prop]{Remark}

\def\begeq{\begin{equation}}
\def\endeq{\end{equation}}

\def\and{\quad{\rm and}\quad}

\def\bl{\bigl(}
\def\br{\bigr)}

\def\<{\langle}
\def\>{\rangle}

\def\Dint{\displaystyle\int}
\def\Dfrac{\displaystyle\frac}

\begin{document}
\title{ Hypersurfaces of Prescribed Curvature Measure}

\author{Pengfei Guan, Junfang Li, and YanYan Li}
\address{Department of Mathematics and Statistics\\
         McGill University\\
         Montreal, Quebec. H3A 2K6, Canada.}
\email{guan@math.mcgill.ca}

\address{Department of Mathematics\\
         University of Alabama at Birmingham\\
         Birmingham, AL 35294}
\email{jfli@uab.edu}

\address{Department of Mathematics\\
         Rutgers University\\
         New Brunswick, NJ 08903}
\email{yyli@math.rutgers.edu}

\thanks{Research of the first author was supported in part by NSERC Discovery Grant. 
Research of the second author was supported in part by NSF DMS-1007223. 
Research of the third author was supported in part by NSF DMS-0701545.}
\begin{abstract}
We consider the corresponding Christoffel-Minkowski problem for curvature measures.  
The existence of star-shaped $(n-k)$-convex bodies with prescribed $k$-th 
curvature measures ($k>0$) has been a longstanding problem. This is settled in this paper 
through the establishment of a crucial $C^2$ a priori estimate for the corresponding 
curvature equation on $\mathbb S^n$.
\end{abstract}
\date{}
\maketitle
\section{Introduction}
This paper concerns the general prescribing curvature measures problem. Curvature measures and area measures are two main subjects in convex geometry. They are the local versions of quermassintegrals in the Brunn-Minkowski theory. The Minkowski problem is a problem of prescribing a given $n$-th area measure. The general Christoffel-Minkowski problem is a problem of prescribing a given
 $k$-th area measure. There is a vast literature devoted to the study of this type
 of problems, and
we refer to  \cite{A1,P1, N, B,CY,F,L,P3, GLy1,
GM} and references therein. We consider the corresponding Christoffel-Minkowski problem
for the curvature measures, that is, the problem of prescribing curvature measures.

We first recall the definition of curvature measures and area measures, more detailed study can be found in
literatures in convex geometry, e.g. Chapter 4 in \cite{sch}. Classically, curvature and area measures were introduced for convex bodies
(nonempty, compact, convex subsets of
$\mathbb R^{n+1}$).
Suppose $K$ is a convex body in $\mathbb R^{n+1}$. There are two notions of local parallel sets: given any Borel set $\beta\in\mathfrak B(\mathbb R^{n+1})$, consider
\[
A_\rho(K,\beta):=\{x\in \mathbb R^{n+1}|0<d(K,x)\le\rho\ \mbox{and}\ p(K,x)\in\beta\}
\]
which is the set of all points $x\in \mathbb R^{n+1}$ for which the distance $d(K,x)\le \rho$ and for which the nearest point $p(K,x)$ belongs to $\beta$. Alternatively, one may prescribe a Borel set $\beta\subset \mathbb S^n$ of unit vectors and then consider
\[
B_\rho(K,\omega):=\{x\in \mathbb R^{n+1}|0<d(K,x)\le\rho\ \mbox{and}\ u(K,x)\in\omega\}
\]
which is the set
of all $x\in \mathbb R^{n+1}$ for which $d(K,x)\le \rho$ and for which the unit vector $u(K,x)$ pointing from $p(K,x)$ to $x$ belongs to $\beta$.

By the theory of convex geometry (e.g., page 203 in \cite{sch}), the measures of the above local parallel sets
 are polynomials in the parameter $\rho$. More precisely,

\begin{equation}
  \begin{array}[]{rll}
    \mathcal H^{n+1}(A_\rho(K,\beta))&=&\displaystyle\frac{1}{n+1}\sum_{m=0}^{n}\rho^{n+1-m}{{n+1}\choose m}\mathcal C_m(K,\beta), \\
    \mathcal H^{n+1}(B_\rho(K,\omega))&=&\displaystyle\frac{1}{n+1}\sum_{m=0}^{n}\rho^{n+1-m}{{n+1}\choose m}S_m(K,\omega),
   \end{array}
    \label{measure}
\end{equation}
for $\beta\in \mathfrak B(\mathbb R^{n+1})$, $\omega\in \mathfrak B(\mathbb S^n)$, and $\rho>0$. The coefficients in the polynomials (\ref{measure}) are defined as generalized curvature measures. Moreover, the measure $\mathcal C_0(K,\cdot),\cdots,\mathcal C_{n}(K,\cdot)$ are called curvature measures of the convex body $K$, and $S_0(K,\cdot), \cdots, S_{n}(K,\cdot)$ are called area measures of $K$.

It is possible to study curvature measures for more general sets. For example, Federer \cite{Fe} introduced the curvature measure for sets of positive reach. Sets of positive reach are generalization of convex sets and smooth submanifolds, for which local parallel sets can be defined in such a way that their measure has a polynomial expansion, yielding a Steiner formula. There has been extensive study on the sets of positive reach and its further generalizations, see the literature in the notes and references of the book \cite{sch}. The basic techniques rely on the geometric measure theory.

The problem of prescribing area measures is called the Christoffel-Minkowski problem. This problem has been extensively studied, we refer to
 \cite{GM} for an updated account. Similar to area measures, the problem of prescribing curvature measures has been discussed in the literature ( e.g. see note $8$ on page $396$ of \cite{sch}). The problem of prescribing $0$-th curvature measure is the Alexandrov problem which is the counterpart of the Minkowski problem. This problem amounts to solving
 a Monge-Amp\'ere type equation on $\mathbb S^n$. The
existence and uniqueness of solutions
 were obtained by A.D. Alexandrov \cite{A2}. For $n=2$ the regularity of
solutions of
 the Alexandrov problem in the
elliptic case was proved by Pogorelov \cite{P2} and for higher dimensional cases, it was solved by Oliker \cite{O}. The general regularity results (degenerate case) of the problem were obtained in \cite{GLy1}.
The general problem of prescribing $k$-th curvature measure for $k\ge0$ is an interesting counterpart of the Christoffel-Minkowski problem. The existence theorem for $k$-th curvature measures ($k>0$) has been open up to now. The main objective of this paper is to settle this problem.

Let us
 start from an equivalent definition of the generalized curvature measures for convex bodies with smooth boundaries. Suppose that
there is a bounded convex domain $\Omega\subset \mathbb R^{n+1}$ with $C^2$ boundary $M$. Let $\kappa=(\kappa_1,\cdots,\kappa_n)$ be the principal curvatures of $M$ at point $x$, $r=(r_1,\cdots,r_n)$ be the principal curvature radii, and $\sigma_k$ be the $k$-th elementary symmetric function. Then the $m$-th curvature measure and area measure of $\Omega$  have the following equivalent form
\begin{eqnarray}\label{def-c}
\begin{array}[]{rll}
  \mathcal C_m(\Omega,\beta)&:=&\displaystyle{n\choose n-m}^{-1}\Dint_{\beta\cap M}\sigma_{n-m}(\kappa)d\mu_g\\
  \mathcal S_m(\Omega,\omega)&:=&\displaystyle{n\choose m}^{-1}\Dint_{\omega}\sigma_{m}(r)d\mathbb S^n,
\end{array}
\end{eqnarray}
for $\beta\in \mathbb R^{n+1}$ and $\omega\in \mathbb S^n$, where $d\mu_g$ is the volume element with respect to the induced metric $g$ of $M$ in $\mathbb R^{n+1}$, and $d\mathbb S^n$ is the volume element of the standard spherical metric.

Note that the classical Christoffel-Minkowski problem for area measures is always confined to convex bodies, as
the area measures are defined through Gauss map on $\mathbb S^n$. For the curvature measure $\mathcal C_m$ defined in (\ref{def-c}), $\Omega$ is
 not necessarily convex if $m>0$. In the work of
  Alexandrov \cite{A2}, the curvature measures is prescribed on $\mathbb S^n$ via radial map. With the radial parameterization of $\Omega$, the natural class is the star-shaped domains.  In the rest of this paper, we treat the prescribing curvature measure problem for bounded star-shaped domains, including the convex bodies as a special case. In short, we prove the existence theorems
 of prescribing general $k$-th curvature measure problem with $k>0$ on bounded $C^2$ star-shaped domains.

Let $\Omega\subset\mathbb R^{n+1}$ be a bounded star-shaped domain with respect to the origin. Assume the boundary $M$ of $\Omega$ is $C^2$. Since $M$ is
star-shaped, it can be viewed as a radial graph of $\mathbb S^n$, i.e.
\[
\begin{array}{rll}
R_M: \mathbb S^n&\longrightarrow M\\
z&\longmapsto \rho(z)z
\end{array}
\]
Following \cite{A2}, up to a normalization constant, for each
star-shaped domain $\Omega$ with $M=\partial \Omega$, the $m$-th curvature measure on each Borel set $\beta$ in $\mathbb S^n$ can be defined as
\[
\mathcal C_m(M,\beta):=\Dint_{R_M(\beta)}\sigma_{n-m}(\kappa)d\mu_g.
\]
In order that the
curvature measure is a regular measure on $\mathbb S^n$, some geometric
conditions on $M$ are necessary.
Recall that the Garding's cone is defined as
\[
\Gamma_k=\{\lambda\in \mathbb R^n| \sigma_i(\lambda)>0, \forall i \le k.\}.
\]
A domain $\Omega\subset \mathbb R^{n+1}$ is called $k$-convex if its principal curvature vector $\kappa(x)=(\kappa_1, \cdots, \kappa_n)\in \Gamma_k$ at every
point $x\in \partial \Omega$.
As the density of the
 $k$-th curvature measure is nonnegative, $M$ is
$(n-k)$-convex. The corresponding Christoffel-Minkowski problem for curvature measures can be formulated as follows.

\medskip

\noindent{\bf Q:} {\it For each $1\le k\le n$ and
 each given positive function $f\in C^2(\mathbb S^n)$, find a closed hypersurface $M$ as a radial
graph over $\mathbb S^n$, such that $\mathcal C_{n-k}(M,\beta)=\int_\beta fd\mu$ for every
 Borel set $\beta$ in $\mathbb S^n$, where $d\mu$ is the standard volume element on $\mathbb S^n$.}

\medskip

As the case $k=n$ is the Alexandrov problem which was completely settled,
we are mainly interested in
 the cases $k<n$. The study of the problem {\bf Q} was initiated by the first and third authors in \cite{GLy}. Recently, with certain  assumptions on $f$, Guan-Lin-Ma obtained the existence and regularity of convex solutions in \cite{GLM}.
The curvature measure problem is equivalent to solving
 a fully nonlinear partial differential equation on $\mathbb S^n$ \cite{GLy, GLM}. For the $C^2$ graph $M$ on $\mathbb S^n$, denote the induced metric to be $g$ and the density function is $\sqrt{\det g}$.  Then
\begin{equation}
\mathcal C_{n-k}(M,\beta)=\Dint_{R_M(\beta)}\sigma_kd\mu_g=\Dint_{\beta}\sigma_k\sqrt{\det g}d\mathbb S^n.
\end{equation}
The above density function can be computed as (see computations in the next section (\ref{density}))
\[
\sqrt{\det g}=\rho^{n-1}\sqrt{\rho+|\nabla \rho|^2},
\]
where the covariant derivative $\nabla$ is with respect to the standard spherical metric. Therefore, we can reduce the prescribing $(n-k)$-th curvature measure problem to the following curvature equation on $\mathbb S^n$:
\begin{equation}\label{cur equ}
\sigma_k(\kappa_1,\cdots,\kappa_n)=\frac{f}{\rho^{n-1}\sqrt{\rho+|\nabla \rho|^2}},
\end{equation}
where $f>0$ is the given function on $\mathbb S^n$. A solution of (\ref{cur equ})  is called admissible if $\kappa(X)\in \Gamma_k$ at each point $X\in M$. We note that any positive $C^2$ function $\rho$ on $\mathbb S^n$ satisfying equation (\ref{cur equ}) is automatically an admissible solution since $f>0$ and
principal curvatures at
a maximum point of $\rho$  are non-negative.
 Equation (\ref{cur equ}) is a special type of fully nonlinear partial differential equations
 studied in the pioneer work by Caffarelli-Nirenberg-Spruck \cite{CNS, CNS85}.

$C^0$ and $C^1$ estimates for admissible solutions
 along with the uniqueness of  admissible solutions
 for equation (\ref{cur equ}) were proved in the unpublished notes \cite{GLy} by the first and third authors. For the existence
of  admissible solutions, the case $k=1$ follows from  theories
 of quasi-linear elliptic
 equations  \cite{GLy}. When $k=n$, admissible solutions are convex and they were dealt with in \cite{P2,O, GLy1}. The existence of admissible solutions for the remaining cases, i.e., $1<k<n$, has been open due to the
lack of  $C^2$ a priori estimate for admissible solutions of (\ref{cur equ}).

We now state the main results. Our first theorem establishes the existence of admissible solutions for $1\le k<n$.

\begin{theo}\label{main thm}
Let $n\ge 2$ and $1\le k\le n-1$. Suppose $f\in C^2(\mathbb S^n)$ and $f>0$. Then there exists a unique $k$-convex star-shaped hypersurface $M\in C^{3,\alpha}$, $\forall \alpha\in (0,1)$ such that it satisfies (\ref{cur equ}). Moreover, there is a constant $C$ depending only on $k, n, \|f\|_{C^{1,1}}, \|1/f\|_{C^0},$
and $ \alpha$ such that,
\begin{equation}\label{c2est} \|\rho\|_{C^{3,\alpha}}\le C.\end{equation}
\end{theo}

It is of interest to find convex solutions to (\ref{cur equ}),
namely, to prove existence results
 for the prescribing curvature measure problem for convex bodies. From the uniqueness of admissible solutions, solutions
 to equation (\ref{cur equ}) are
 not convex in general, except for the case
 $k=n$.  For  $k<n$, there has been some progress recently. In \cite{GLM}, Guan-Lin-Ma proved $C^2$ estimates and existence theorems for convex solutions of equation (\ref{cur equ}) for $1\le k<n$ under suitable convexity
conditions of $f$. More specifically, they proved

\noindent{\bf Theorem A. }{\it {(Curvature measure problem under strict convexity condition \cite{GLM})}
Suppose $f\in  C^2(\mathbb S^n)$,
$f>0$, $n\ge2$, $1\le k\le n-1$. If $f$ satisfies
that
\begin{equation}\label{strict convex cond}
|X|^\frac{n+1}{k}f\big( \frac{X}{|X|} \big)^{-\frac{1}{k}} {\rm\ is\ a\ strictly\ convex\ function\ in\  } \mathbb R^{n+1}\backslash \{ 0\},
\end{equation}
then there exists a unique strictly convex hypersurface $M\in C^{3,\alpha}$, $\alpha\in (0,1)$ such that it satisfies (\ref{cur equ}).
}

\medskip

Moreover, if $k=1$, or $2$, i.e., the cases of mean curvature measure and scalar curvature measure, condition (\ref{strict convex cond}) can be weakened as below.

\noindent{\bf Theorem B. }{\em {(Mean curvature measure and scalar curvature measure under convexity condition \cite{GLM})}
Suppose $k=1$, or $2$, and $k<n$. Suppose $f\in  C^2(\mathbb S^n)$ is a positive function. If $f$ satisfies that
\begin{equation}\label{convex cond}
|X|^\frac{n+1}{k}f\big( \frac{X}{|X|} \big)^{-\frac{1}{k}} {\rm\ is\ a\ convex\ function\ in\  } \mathbb R^{n+1}\backslash \{ 0\},
\end{equation}
then there exists a unique strictly convex hypersurface $M\in C^{3,\alpha}$, $\alpha\in (0,1)$ such that it satisfies (\ref{cur equ}).
}

\bigskip

With the a priori estimates for admissible solutions of equation (\ref{cur equ}) in Theorem \ref{main thm}, we obtain the existence of convex bodies under weaker condition (\ref{convex cond}) for all $1\le k<n$.

\begin{theo}\label{convexity thm}
Suppose $1\le k< n$ and $f\in  C^2(\mathbb S^n)$ is a positive function. If $f$ satisfies
\begin{equation}\label{convex cond1}
|X|^\frac{n+1}{k}f\big( \frac{X}{|X|} \big)^{-\frac{1}{k}} {\rm\ is\ a\ convex\ function\ in\  } \mathbb R^{n+1}\backslash \{ 0\},
\end{equation}
then there exists a unique strictly convex hypersurface $M\in C^{3,\alpha}$, $\alpha\in (0,1)$ such that it satisfies (\ref{cur equ}).
\end{theo}

\bigskip

The rest of this paper is organized as follows. In Section 2, we fix notations and list necessary formulas we need. In Section 3, we establish the crucial $C^2$ estimate and prove Theorem \ref{main thm} and Theorem \ref{convexity thm}.

\section{Preliminaries}
We first recall the relevant geometric quantities for a smooth closed hypersurface in $\mathbb R^{n+1}$ we may need. Throughout the paper, repeated indices denote summation and we assume the origin is inside the domain enclosed by $M$.

Let $M^n$ be an immersed hypersurface in $\mathbb R^{n+1}$. For $X\in M\subset \mathbb R^{n+1}$, choose local normal coordinates in $\mathbb R^{n+1}$, such that $\{ \frac{\partial}{\partial x_1}, \cdots, \frac{\partial}{\partial x_n} \}$ are tangent to $M$ and $\partial_{n+1}$ is the unit outer normal of the hypersurface. We sometimes denote $\partial_i:=\frac{\partial}{\partial x_i}$ and also use $\nu$ to denote the unit outer normal $\partial_{n+1}$. We use lower indices to denote covariant derivatives with respect to the induced metric.

For the immersion $X$, the second fundamental form is the symmetric $(2,0)$-tensor given by the matrix $\{h_{ij}\}$,
\begin{equation}
h_{ij}=\langle\partial_iX,\partial_j\nu\rangle.
\end{equation}

Recall the following identities:
\begin{equation}
\begin{array}{rll}
X_{ij}=& -h_{ij}\nu\quad {\rm (Gauss\ formula)}\\
(\nu)_i=&h_{ij}\partial_j\quad {\rm (Weigarten\ equation)}\\
h_{ijk}=& h_{ikj}\quad {\rm (Codazzi\ formula)}\\
R_{ijkl}=&h_{ik}h_{jl}-h_{il}h_{jk}\quad {\rm (Gauss\ equation)},\\
\end{array}
\end{equation}
where $R_{ijkl}$ is the $(4,0)$-Riemannian curvature tensor. We also have
\begin{equation}
\begin{array}{rll}
h_{ijkl}=& h_{ijlk}+h_{mj}R_{imlk}+h_{im}R_{jmlk}\\
=& h_{klij}+(h_{mj}h_{il}-h_{ml}h_{ij})h_{mk}+(h_{mj}h_{kl}-h_{ml}h_{kj})h_{mi}.\\
\end{array}
\end{equation}

In certain cases of this article, we need to carry out calculations in a neighborhood of the standard sphere $\mathbb S^n$. Since $M$ is star-shaped with respect to the origin, the position vector $X$ can be written as $X(x)=\rho(x)x$, $x\in \mathbb S^n$ for some smooth function $\rho$ on $ \mathbb S^n$. In this case, suppose $\{ \partial_1, \cdots, \partial_n \}$ is
some local normal coordinates on $\mathbb S^n$ and $\nabla$ is the covariant differentiation with respect to the standard metric on  $\mathbb S^n$. Then the induced metric $g_{ij}$ of $M$ is given by
\begin{equation}
g_{ij}=\rho^2\delta_{ij}+\rho_i\rho_j,
\end{equation}
and the area density function
\begin{equation}\label{density}
\sqrt{\det g}=\rho^{n-1}\sqrt{\rho^2+|\nabla \rho|^2}.
\end{equation}
The second fundamental form of $M$ can be calculated as
\begin{equation}\label{2nd FF rho}
h_{ij}=(\rho^2+|\nabla \rho|^2)^\frac{1}{2}(\rho^2\delta_{ij}+2\rho_i\rho_j-\rho\rho_{ij})
\end{equation}
and the unit outer normal vector is
\begin{equation}\label{unit normal}
\nu=\frac{\rho x-\nabla \rho}{\sqrt{\rho^2+|\nabla \rho|^2}}.
\end{equation}
The support function of a hypersurface is defined as $u(X)=\langle X, \nu\rangle$. We can also compute it as
\begin{equation}\label{}
u=\rho^2(\rho^2+|\nabla \rho|^2)^{-\frac{1}{2}}.
\end{equation}

The principal curvatures $(\kappa_1, \cdots, \kappa_n)$ are the eigenvalues of the second fundamental form with respect to the metric which satisfy
\[
\det(h_{ij}-\kappa g_{ij})=0.
\]
The curvature equation (\ref{cur equ}) on $\mathbb S^n$ can also be equivalently expressed as an equation of the position vector $X$. Using (\ref{unit normal}), we have
\begin{equation}\label{cur equ 2}
\sigma_k(\kappa_1,\cdots,\kappa_n)(X)=\frac{u(X)}{|X|^{n+1}}f\big(\frac{X}{|X|}\big),\quad \forall X\in M.
\end{equation}

\bigskip

The uniqueness results, $C^0$, $C^1$ estimates for solutions of equation (\ref{cur equ}) were proved in \cite{GLy} (see also \cite{GLM}). For completeness, we state the results here. The first lemma concerns the uniqueness of solutions.

\begin{lemm}\label{lemma uniqueness}
Let $1\le k<n$. Suppose $\rho_1$, $\rho_2$ are two solutions of equation (\ref{cur equ}) and $\lambda(\rho_i)\in \Gamma_k$, for $i=1,2$. Then $\rho_1\equiv \rho_2$.
\end{lemm}

The following lemma will be useful later in this paper. Set $F(\lambda)=\sigma_k(\lambda)^\frac{1}{k}$. $F(\lambda) $ is homogeneous of degree one. Equation (\ref{cur equ})  can be written as
\[
F(\lambda)\equiv F(\lambda_1,\cdots,\lambda_n)=f^\frac{1}{k}\rho^{\frac{(1-n)}{k}}(\rho^2+|\nabla \rho|^2)^{-1/(2k)}\equiv K(x,\rho,\nabla \rho).
\]

\begin{lemm}\label{lemma openness}
Let $L$ denote the linearized operator of $F(\lambda)-K(x,\rho,\nabla\rho)$ at a solution $\rho$ of (\ref{cur equ}). If $\omega$ satisfies $L\omega=0$ on $\mathbb S^n$, then $\omega\equiv0$ on $\mathbb S^n$.
\end{lemm}

Next is the $C^0$, $C^1$ estimates.

\begin{lemm}\label{c1}
If $M$ satisfies (\ref{cur equ 2}), then
\[
\Big( \frac{\min_{\mathbb S^n}f}{C^k_n}  \Big)^\frac{1}{n-k}\le \min _{\mathbb S^n}|X|\le \max _{\mathbb S^n}|X|\le\Big( \frac{\max_{\mathbb S^n}f}{C^k_n}  \Big)^\frac{1}{n-k}.
\]
Moreover, there exits a constant $C$ depending only on $n$, $k$, $\min_{\mathbb S^n}f$, $|f|_{C^1}$ such that
\[
\max_{\mathbb S^n}|\nabla \rho|\le C.
\]
\end{lemm}

\section{the $C^2$ a priori estimates and the proof of main results}

The major step of this paper is to establish the a priori $C^2$-estimates. Equation (\ref{cur equ}) is similar to the prescribing curvature equation treated
 in \cite{CNS85}. There, $C^2$ estimates were
 proved for a general type of curvature equations
 using the concavity of $\sigma^{\frac1k}$. Equation (\ref{cur equ}) differs from the equations treated in \cite{CNS85}, as the right hand side of (\ref{cur equ}) depends on $\nabla \rho$ too.  We do not see how to
apply the arguments in \cite{CNS85} to
 deal with equation (\ref{cur equ}). Under certain structural convexity conditions, $C^2$-estimates were obtained for convex solutions of (\ref{cur equ}) in \cite{GLM}. For admissible solutions, this has been a missing piece for a while. In this section, we prove a key Lemma \ref{key lemma} to overcome the main difficulty in the proof of $C^2$ estimate for admissible solutions. We believe
that  Lemma \ref{key lemma} will be useful to treat other curvature problems as well.
\begin{prop}\label{C2 proposition}
For $1<k<n$, let $F\equiv \sigma_k=\phi u$ and denote $H\equiv \sigma_1$, then at a
maximum  point of $\frac{H}{u}$,
\begin{equation}\label{C2 identity}
\begin{array}{rll}

F^{ij}\bl\frac{H}{u}\br _{ij}
=& \frac{1}{u}[\phi_{ss}u+2\phi_su_s]-\bl\frac{H}{u}\br \phi_l\<X,X_l\>-(k-1)\bl\frac{H}{u}\br \phi\\
&
+ (k-1)\phi|A|^2-\frac{1}{u}F^{ij;ml}h_{ij;s}h_{ml;s},\\
\end{array}
\end{equation}
where $A$ denotes the second fundamental form.
\end{prop}

\begin{proof}
We will need the following formulas.

\begin{equation}\label{formulas C2}
\begin{array}{rll}
u_s=&h_{sl}\<X,X_l\>\\
u_{ij}=&h_{ij;l}\<X,X_l\>+h_{ij}-(h^2)_{ij}u\\
h_{ij;kl}=&h_{kl;ij}+(h_{lk}h_{im}-h_{lm}h_{ik})h_{mj}+(h_{lj}h_{im}-h_{lm}h_{ij})h_{mk}\\
F^{ij}h_{ij;st}=&F_{st}-F^{ij;ml}h_{ml;s}h_{ij;t}.\\
\end{array}
\end{equation}

Assume that
 $\frac{H}{u}$ achieves the maximal value at some
point $P\in M^n$ which implies $\bl\frac{H}{u}\br_i(P)=0$.
In the rest of the proof, all calculations will be performed at the
maximum  point $P$. First, since $\bl\frac{H}{u}\br_i=0$, we have

\begin{equation}
\begin{array}{rll}
F^{ij}\bl\frac{H}{u}\br_{ij}=& F^{ij}\bigg[\frac{H_{ij}}{u}-\frac{u_{j}}{u}\bl\frac{H}{u}\br_{i}-\frac{u_{i}}{u}\bl\frac{H}{u}\br_{j}
-\bl\frac{H}{u}\br\frac{u_{ij}}{u}\bigg]\\
=&\frac{1}{u}F^{ij}H_{ij}-\frac{1}{u}\bl\frac{H}{u}\br F^{ij}u_{ij}.

\end{array}
\end{equation}

On the other hand, applying formulas in (\ref{formulas C2}), we have

\begin{equation}\label{C2 equ1}
\begin{array}{rll}
\frac{1}{u}F^{ij}H_{ij}
=& \frac{1}{u}F^{ij}h_{ss;ij}\\
=& \frac{1}{u}F^{ij}\big[ h_{ij;ss}+(h_{ij}h_{sm}-h_{jm}h_{si})h_{ms}+(h_{js}h_{sm}-h_{jm}h_{ss})h_{mi} \big] \\
=& \frac{1}{u}F^{ij}h_{ij;ss}+ k\phi|A|^2- \frac{1}{u}F^{ij}(h^2)_{ij}H\\
=& \frac{1}{u}F_{ss}-\frac{1}{u}F^{ij;ml}h_{ij;s}h_{ml;s}+ k\phi|A|^2- \bl\frac{H}{u}\br F^{ij}(h^2)_{ij}\\
=& \frac{1}{u}[\phi_{ss}u+2\phi_su_s+\phi u_{ss}]-\frac{1}{u}F^{ij;ml}h_{ij;s}h_{ml;s}+ k\phi|A|^2\\
&- \bl\frac{H}{u}\br F^{ij}(h^2)_{ij}\\
=& \frac{1}{u}[\phi_{ss}u+2\phi_su_s]+\frac{\phi}{u}\big[H_l\<X,X_l\>+H-|A|^2u \big]\\
&-\frac{1}{u}F^{ij;ml}h_{ij;s}h_{ml;s}+ k\phi|A|^2- \bl\frac{H}{u}\br F^{ij}(h^2)_{ij}\\
=& \frac{1}{u}[\phi_{ss}u+2\phi_su_s]+\frac{\phi}{u}H_l\<X,X_l\>+\bl\frac{H}{u}\br \phi\\
&-\frac{1}{u}F^{ij;ml}h_{ij;s}h_{ml;s}+ (k-1)\phi|A|^2- \bl\frac{H}{u}\br F^{ij}(h^2)_{ij}.
\end{array}
\end{equation}
We also compute the following
\begin{equation}\label{C2 equ2}
\begin{array}{rll}

-\frac{1}{u}\bl\frac{H}{u}\br F^{ij}u_{ij}=&-\frac{1}{u}\bl\frac{H}{u}\br F^{ij}\bigg[ h_{ij;l}\<X,X_l\>+h_{ij}-(h^2)_{ij}u \bigg]\\
=&-\frac{1}{u}\bl\frac{H}{u}\br F_l\<X,X_l\>-k\phi\bl\frac{H}{u}\br +\bl\frac{H}{u}\br F^{ij}(h^2)_{ij}\\
=&-\frac{\phi}{u}\bl\frac{H}{u}\br  u_l\<X,X_l\>-\bl\frac{H}{u}\br \phi_l\<X,X_l\>-k\phi\bl\frac{H}{u}\br +\bl\frac{H}{u}\br F^{ij}(h^2)_{ij},\\

\end{array}
\end{equation}
where $(h^2)_{ij}=h_{ik}h_{kj}$.

Adding up (\ref{C2 equ1}) and (\ref{C2 equ2}),
and using the critical point condition, we obtain

\begin{equation}\label{}
\begin{array}{rll}

F^{ij}\bl\frac{H}{u}\br _{ij}
=& \frac{1}{u}[\phi_{ss}u+2\phi_su_s]+\phi \bl\frac{H}{u}\br _l\<X,X_l\>-\bl\frac{H}{u}\br \phi_l\<X,X_l\>\\
&-(k-1)\bl\frac{H}{u}\br \phi
-\frac{1}{u}F^{ij;ml}h_{ij;s}h_{ml;s}+ (k-1)\phi|A|^2\\
\\

=& \frac{1}{u}[\phi_{ss}u+2\phi_su_s]-\bl\frac{H}{u}\br \phi_l\<X,X_l\>-(k-1)\bl\frac{H}{u}\br \phi\\
&-\frac{1}{u}F^{ij;ml}h_{ij;s}h_{ml;s}+ (k-1)\phi|A|^2,
\end{array}
\end{equation}
which finishes the proof.
\end{proof}

The following lemma is the key for $C^2$-estimates.
\begin{lemm}\label{key lemma}
Let $\alpha=\frac{1}{k-1}$, if $(h_{ij})\in \Gamma_k$, then
\begin{equation}\label{key lemma equ2}
{(\sigma_k)^{ij,lm}h_{ij;s}h_{ij;s}}\le-{\sigma_k}\bigg[ \frac{(\sigma_k)_s}{\sigma_k} -\frac{(\sigma_1)_s}{\sigma_1} \bigg]\bigg[ (\alpha-1)\frac{(\sigma_k)_s}{\sigma_k}
 -(\alpha+1)\frac{(\sigma_1)_s}{\sigma_1}  \bigg].
\end{equation}
\end{lemm}

\begin{proof}
By the concavity of $\bigg(\displaystyle \frac{\sigma_k}{\sigma_1}\bigg)^\frac{1}{k-1}(h_{ij})$, we have
\begin{equation}\label{key lemma equ3}
\begin{array}{rll}
0\ge&\displaystyle \frac{\partial^2 }{\partial h_{ij}\partial h_{lm}}\displaystyle\bl\bigg( \frac{\sigma_k}{\sigma_1}\bigg)^\frac{1}{k-1}\br h_{ij;s}h_{lm;s}.

\end{array}
\end{equation}
To simplify notations, we denote $\alpha=\frac{1}{k-1}$. Direct computations yield,

\begin{equation}\label{key lemma equ4}
\begin{array}{rll}
0\ge&  \Dfrac{\partial^2 }{\partial h_{ij}\partial h_{lm}}\bigg(  \frac{\sigma_k}{\sigma_1}\bigg)^\alpha\cdot h_{ij;s}h_{lm;s}\\
\\
=&\alpha\bigg(  \frac{\sigma_k}{\sigma_1}\bigg)^\alpha
\bigg[\frac{(\sigma_k)^{ij,lm}}{\sigma_k}+\frac{(\alpha-1)(\sigma_k)^{ij}(\sigma_k)^{lm}}{\sigma_k^2}\\
&-\frac{2\alpha(\sigma_k)^{ij}(\sigma_1)^{lm}}{\sigma_k\sigma_1}
+\frac{(\alpha+1)(\sigma_1)^{ij}(\sigma_1)^{lm}}{\sigma_1^2}\bigg] h_{ij;s}h_{lm;s}\\
\end{array}
\end{equation}

Equivalently,
\begin{equation}\label{key lemma equ5}
\begin{array}{rll}
\frac{(\sigma_k)^{ij,lm} h_{ij;s}h_{lm;s}}{\sigma_k}\le&-\bigg[\frac{(\alpha-1)(\sigma_k)^{ij}(\sigma_k)^{lm}}{\sigma_k^2}
-\frac{2\alpha(\sigma_k)^{ij}(\sigma_1)^{lm}}{\sigma_k\sigma_1}\\
&+\frac{(\alpha+1)(\sigma_1)^{ij}(\sigma_1)^{lm}}{\sigma_1^2}\bigg] h_{ij;s}h_{lm;s}\\
\\
\le&-\bigg[ \frac{(\sigma_k)_s}{\sigma_k} -\frac{(\sigma_1)_s}{\sigma_1} \bigg]\bigg[ (\alpha-1)\frac{(\sigma_k)_s}{\sigma_k}
 -(\alpha+1)\frac{(\sigma_1)_s}{\sigma_1}  \bigg]
\end{array}
\end{equation}

\end{proof}

Note in Lemma \ref{key lemma}, one may replace $\sigma_k$ by any positive function $F$ with the property that
$(\frac{F}{\sigma_1})^{\alpha}$ is concave for some $\alpha>0$.
The following is a corollary of Lemma \ref{key lemma}.
\begin{coro}\label{key lemma coro}
If $\frac{(\sigma_1)_s}{\sigma_1}=\frac{(\sigma_k)_s}{\sigma_k} -r$ for some $r$,
\begin{equation}\label{key lemma coro equ2}
{(\sigma_k)^{ij,lm}h_{ij;s}h_{ij;s}}\le \max\bigg\{2r(\sigma_k)_s-\frac{k}{k-1}r^2\sigma_k,0\bigg\}
.
\end{equation}
\end{coro}

Combining Proposition \ref{C2 proposition} and Corollary \ref{key lemma coro}, we obtain the following $C^2$ estimates.

\begin{lemm}\label{Lemma C2}
If $M$ satisfies equation (\ref{cur equ 2}) and $2\le k\le n$, then there exists a constant $C$ depending only on $n$, $k$, $\min_{S^n}f$, $|f|_{C^1}$, and $|f|_{C^2}$, such that
\begin{eqnarray}\label{C2 Lemma inequ}
\max_{M}\sigma_1\le C,  \quad \nabla^2 \rho\le C.\end{eqnarray}
\end{lemm}

\begin{proof}
We have already obtained the $C^0$ and $C^1$ estimates for $\rho$. Therefore, to bound $\nabla^2\rho$, one only need to obtain
 an estimate for $\sigma_1(h^i_j)=g^{ij}h_{ij}$ in view of $\sigma_1^2=|A|^2+2\sigma_2$ and $\sigma_2>0$, where $h^i_j=g^{ik}h_{kj}$ is the Weingarten tensor.

If we denote $H\doteq \sigma_1$, then at a
 point $P$ where $\frac{H}{u}$ achieves its maximum, we have shown in (\ref{C2 identity})
\begin{equation}\label{C2 inequality}
\begin{array}{rll}

0\ge&F^{ij}\bl\frac{H}{u}\br _{ij}\\

=& \frac{1}{u}[\phi_{ss}u+2\phi_su_s]-\bl\frac{H}{u}\br \phi_l\<X,X_l\>-(k-1)\bl\frac{H}{u}\br \phi\\
&-\frac{1}{u}F^{ij;ml}h_{ij;s}h_{ml;s}+ (k-1)\phi|A|^2.
\end{array}
\end{equation}

Recall that $\phi(X)\in C^2(M)$ is defined as $\phi(X)=|X|^{-(n+1)}f(\frac{X}{|X|})$ and $C^0$, $C^1$ estimates of $\rho=|X|$ are already known. We have the following estimates.
\[
\begin{array}{rll}
|\phi_i|(P)\le& C(n, k, \min_{S^n}f, |f|_{C^1})
\\
|\phi_{ii}|(P)\le &C(n, k, \min_{S^n}f, |f|_{C^1},|f|_{C^2})\big(1+|A|(P)\big)
\end{array}
\]

On the other hand, $|u_i|=|h^i_j\rho\rho_j|\le c_3|A|$.
Observe that by equation (\ref{cur equ 2}), $\frac{\sigma_1}{u}=\frac{\sigma_1\phi}{\sigma_k}$. At a maximum point of the test function $\frac{\sigma_1}{u}$, one has $\frac{(\sigma_1)_s}{\sigma_1}=\frac{(\sigma_k)_s}{\sigma_k}-\frac{\phi_s}{\phi}$. In Corollary \ref{key lemma coro}, let $r=\frac{\phi_s}{\phi}(P)$, then
\[
\begin{array}{rll}F^{ij;ml}h_{ij;s}h_{ml;s}\le& 2r(u\phi)_{s}-\frac{k}{k-1}r^2u\phi\\
\le& C_1(n, k, \min_{S^n}f, |f|_{C^1})|A|+C_2(n, k, \min_{S^n}f, |f|_{C^1}).
\end{array}
\]

With the above estimates, we simplify (\ref{C2 inequality}) to be
\begin{equation}\label{}
\begin{array}{rll}

|A|^2(P)+c_4|A|(P)+c_5\le 0,
\end{array}
\end{equation}
where $c_4$ and $c_5$ are constants depending only on  $n$, $k$, $\min_{S^n}\phi$, $|f|_{C^1}$, and $|f|_{C^2}$.
Hence at $P$, $|A|(P)\le C$. In turn $\sigma_1(X)\le u(X)\frac{\sigma_1(P)}{u(P)}\le C$ for any $X\in M$. This implies (\ref{C2 Lemma inequ}). \end{proof}

\bigskip

\begin{rema} At this point, we would like to raise a question regarding global $C^2$ estimates for general curvature equations. To be precise, suppose $M\subset \mathbb R^{n+1}$ is a compact smooth hypersurface satisfying equation
\begin{equation} \sigma_k(\kappa(X))=f(X, \nu),
\kappa(X)\in \Gamma_k,\ \ \ \ \forall X\in M,\end{equation}
where $f\in C^2(\mathbb R^{n+1}\times \mathbb S^n)$ is a general positive function. Suppose there is a priori $C^1$ bound of $M$, can one conclude a $C^2$
a priori bound of $M$ in terms of $C^1$ norm of $M$, $f$, $n,k$? When $k=1$, the answer is affirmative. It follows from theories
 of quasilinear elliptic equations, in particular that of
 the mean curvature type equation. When $k=n$, this is also true (e.g., Theorem 5.1 in \cite{GLM}), it is parallel to a well known fact for Monge-Amp\`ere equation.

A similar question can also be asked for Hessian equations ($1<k<n$).
 Suppose that $u$ is an admissible solution of equation
\begin{equation} \sigma_k(\nabla^2 u)=f(x, u, \nabla u), \forall x\in \Omega \subset \mathbb R^n,\end{equation}
where $f\in C^2(\Omega \times \mathbb R \times \mathbb R^n)$ is a general positive function, is it true that there is $C$ depending only on $n, k, \|f\|_{C^{1,1}}, \|1/f\|_{C^0}, \|u\|_{C^{0,1}(\bar \Omega)}$ such that
\begin{equation}\label{c2hyp} \max_{x\in \Omega}|\nabla^2 u(x)|\le C
\big(1+ \max_{x\in \partial \Omega}|\nabla^2u(x)|\big)?\end{equation}
\end{rema}

\bigskip

Now we are ready to prove the main theorems of this paper.

\noindent {\bf Proof of Theorem \ref{main thm}.}
For any positive function $f\in C^2(\mathbb S^n)$, for $0\le t\le 1$ and $1< k< n-1$, set $f_t(x)=[1-t+tf^{-\frac{1}{k}}(x)]^{-k}$. We consider the following family of equations for $0\le t\le 1$:
\begin{equation}\label{continuity method}
\sigma_k(\kappa_1,\cdots,\kappa_n)(x)=f_t(x)\rho^{1-n}(\rho^2+|\nabla \rho|^2)^{-1/2}, \quad {\rm on\ }\mathbb S^n,
\end{equation}
where $n\ge2$. We want to find admissible solutions in the class of star-shaped hypersurfaces. Let $I=\{t\in [0,1]: {\rm such\ that\ }
(\ref{continuity method})\  {\rm is\ solvable}\}$. $I$ is nonempty because $\rho=[C^k_n]^{-\frac{1}{n-2}}$ is a solution for $t=0$. By the a priori estimates of lemmas \ref{c1} and \ref{Lemma C2}, the Evans-Krylov theorem and the
Schauder theorem, $\rho\in C^{3,\alpha}(\mathbb S^n)$ and
\[
\|\rho\|_{C^{3,\alpha}(\mathbb S^n)}\le C,
\]
where $C$ depends only on  only on $n$, $k$, $\min_{S^n}f$, $\max_{S^n}f$, $|f|_{C^1}$, $|f|_{C^2}$ and $\alpha$. The a priori estimates guarantee
that  $I$ is closed. The openness comes from Lemma \ref{lemma openness} and
the implicit function theorem. This proves the existence
part of the theorem. The uniqueness part of the theorem follows
from  Lemma \ref{lemma uniqueness}.
The proof of  Theorem \ref{main thm} is completed. \qed

\bigskip

To prove Theorem \ref{convexity thm}, we need the following Constant Rank Theorem in \cite{GLM}.
\begin{theo}\label{GLM thm}
\cite{GLM}  Suppose $M$ is a convex hypersurface and satisfies equation (\ref{cur equ 2}) for $k<n$ with the second fundamental form $W=\{h_{ij}\}$ and $|X|^\frac{n+1}{n}f(\frac{X}{|X|})$ is convex in $\mathbb R^{n+1}\backslash \{0\}$, then $W$ is positive definite.
\end{theo}

\noindent{\bf Proof of Theorem \ref{convexity thm}.}
By the proof of Theorem \ref{main thm}, there exits a unique admissible solution to (\ref{cur equ 2}). We need to prove the strict convexity of
 solutions under the weak convexity condition (\ref{convex cond1}) of $|X|^\frac{n+1}{k}f\big( \frac{X}{|X|} \big)^{-\frac{1}{k}}$. Using the same deformation path as in the proof of Theorem \ref{main thm}, we start from the sphere solution with $\rho=[C^k_n]^{\frac{1}{n-2}}$ and
prove the theorem by contradiction. Suppose that
 the strict convexity is not true. By continuity, there must exist a
 first $t_0\in[0,1]$ such that the second fundamental form $W(x,t)>0$ for all $x\in M$, $t\in [0,t_0)$. Moreover, at $t=t_0$, $W(x,t_0)\ge0$ for all $x\in M$ and there exists $x_0\in M$ such that $W(x_0,t_0)=0$. However, since the hypersurface $M(t_0)$ is convex, by Theorem \ref{GLM thm}, $M(x,t_0)$ must be strictly convex for any $x\in M(t_0)$ which contradicts $W(x_0,t_0)=0$. This proves that the solution hypersurface to (\ref{cur equ 2}) must be strictly convex. \qed

\bigskip

The condition in Theorem \ref{convexity thm} can be weakened further following the same compactness argument in the proof of Theorem 5 in \cite{gm1}.

\begin{theo}\label{convexity thm1}
Suppose $1\le k< n$ and $f\in  C^2(\mathbb S^n)$ is a positive function. There is $\delta>0$ depending
only on $n, k, \|f\|_{C^{1,1}}, \|1/f\|_{C^0}$ such that, if $f$ satisfies
\begin{equation}\label{convex cond11}
|X|^\frac{n+1}{k}f\big( \frac{X}{|X|} \big)^{-\frac{1}{k}} +\delta |X|^2 {\rm\ is\ a\ convex\ function\ in\  } \mathbb R^{n+1}\backslash \{ 0\},
\end{equation}
then there exists a unique strictly convex hypersurface $M\in C^{3,\alpha}$, $\alpha\in (0,1)$ such that it satisfies (\ref{cur equ}).
\end{theo}

\noindent {\bf Proof.} Suppose it is not true. In view of Theorem \ref{main thm}, there is a constant $C>0$ and a sequence of positive functions $f_l$ on $\mathbb S^n$ with
\begin{eqnarray*} && \|f_l\|_{C^{1,1}}\le C, \quad \|\frac1{f_l}\|_{C^{0}}\le C,\\
&& |X|^\frac{n+1}{k}f_l\big( \frac{X}{|X|} \big)^{-\frac{1}{k}} +\frac1l |X|^2 {\rm\  convex\  in\  } \mathbb R^{n+1}\backslash \{ 0\},\end{eqnarray*}
and a sequence of admissible solutions $M_l=\{X=\rho_l(\frac{X}{|X|})X\}$ with principal curvature vector
$\kappa^l$ satisfying
\[\sigma_k(\kappa^l_1,\cdots,\kappa^l_n)=\frac{f_l}{\rho_l^{n-1}\sqrt{\rho_l+|\nabla \rho_l|^2}},\]
and a sequence of points $x_l\in \mathbb S^n$ such that
\begin{equation}\label{kappa} \kappa^l_{j(l)}(x_l)\le 0.\end{equation}
Again by Theorem \ref{main thm}, there is
some $\tilde C$ depending on on $k, n, \|f_l\|, \|\frac1{f_l}\|_{C^{0}}, \alpha$ such that
\[\|\rho\|_{C^{3,\alpha}}\le \tilde C.\]
By the compactness, there exist subsequences which we still denote $x_{l}, f_{l}, \rho_{l}$ such that
\[ x_{l}\to x_0\in \mathbb S^n, \quad f_{l}\to f \in C^{1,\alpha}, \quad \rho_{l}\to \rho \in C^{3,\alpha}\]
with $\rho$ satisfying equation (\ref{cur equ}). Clearly, $f\in C^{1,1}$ and it satisfies condition (\ref{convex cond1}). By Theorem \ref{convexity thm}, $M$ is strictly convex. But from (\ref{kappa}), $M$ is not strictly convex at $x_0$. This is a contradiction. \qed

\end{document}